\documentclass{amsart}

\newtheorem{theorem}{Theorem}[section]
\newtheorem*{theorem A}{Theorem A}
\newtheorem*{theorem B}{N\"olker's Theorem}
\newtheorem{lemma}{Lemma}[section]

\newtheorem{corollary}{Corollary}[section]

\theoremstyle{remark}
\newtheorem{remark}{Remark}[section]
\theoremstyle{remark}

\theoremstyle{definition}

\newtheorem{definition}{Definition}[section]

\numberwithin{equation}{section}
\def\({\left ( }
\def\){\right )}
\def\<{\left < }
\def\>{\right >}

\begin{document}

\noindent {\sc {Mathematical Sciences And Applications E-Notes}}

\noindent {\sc \small Volume 1  No. 1 pp. 000--000 (2013) \copyright
MSAEN}

\vspace{2cm}

\title{Cartan-type criterions for constancy \\ of almost Hermitian manifolds}

\author{Hakan Mete Ta\c stan}
\address{Department of Mathematics, \.Istanbul University, Vezneciler, 34134, \.Istanbul, TURKEY}
\email{hakmete@istanbul.edu.tr}




\dedicatory{{\rm (Communicated by Bayram \d{S}AH\.{I}N)}}


\begin{abstract}
We studied the axiom of anti-invariant 2-spheres and the axiom of
co-holomorphic $(2n+1)$-spheres. We proved that a nearly
K\"{a}hlerian manifold satisfying the axiom of anti-invariant
2-spheres is a space of constant holomorphic sectional curvature. We
also showed that an almost Hermitian manifold $M$ of dimension
$2m\geq6$ satisfying the axiom of co-holomorphic $(2n+1)$-spheres
for some $n,$  where $(1\leq n\leq m-1)$, the manifold $M$ has
pointwise constant type $\alpha$ if and only if $M$ has pointwise
constant anti-holomorphic sectional curvature $\alpha$.
\end{abstract}
\maketitle
\section{Introduction}
E. Cartan \cite{Car} introduced the axiom of $n$-planes as: A
Riemannian manifold $M$ of dimension $m\geq3$ is said to satisfy the
axiom of $n$-planes, where $n$ is a fixed integer $2\leq n\leq m-1$,
if for each point $p\in M$ and each $n$-dimensional subspace
$\sigma$ of the tangent space $T_p(M)$, there exists an
$n$-dimensional totally geodesic submanifold $N$ such that $p\in N$
and $T_p(N)=\sigma.$ He also gave a criterion for constancy of
sectional curvature for any Riemannian manifold of dimension
$m\geq3$ in the following theorem.
\begin{theorem} A Riemannian manifold of dimension $m\geq3$ with the axiom of $n$-planes is a real space form.
\end{theorem}
D.S. Leung and K. Nomizu \cite{Leu} introduced \emph{the axiom of
$n$-spheres} by using totally umbilical submanifold $N$ with
parallel mean curvature vector field instead of
totally geodesic submanifold $N$ in the axiom of $n$-planes. They proved a generalization of Theorem 1.1.\\

Later on, Cartan's idea was applied to almost Hermitian manifolds in
various studies. K\"{a}hlerian manifolds were studied in
\cite{BChe,Gol,Ha,Kassa,No,Yama,Yano}. The articles \cite{Vanh} and
\cite{Yam} discussed nearly K\"{a}hlerian (almost Tachibana)
manifolds. The results concerning larger classes of almost Hermitian
manifolds can be found in \cite{Kas,Kass,Tas,Van,Vanh}. Here, we
shall call the criterions used in all of the above papers as
\emph{Cartan-type criterions}.
\section{Preliminaries}

\subsection{Some classes of almost Hermitian manifolds}
Let $M$ be an almost Hermitian manifold with an almost complex
structure $J$ in its tangent bundle and a Riemannian metric $g$ such
that $g(JX,JY)=g(X,Y)$ for all $X,Y\in\chi(M)$, where $\chi(M)$ is
the Lie algebra of $C^{\infty}$ vector fields on $M$. Let $\nabla$
be the Riemannian connection on $M$. The Riemannian  curvature
tensor $R$ associated with $\nabla$ is defined by
$R(X,Y)=\nabla_{[X,Y]}-[\nabla_{X},\nabla_{Y}].$ We denote
$g(R(X,Y)Z,U)$ by $R(X,Y,Z,U)$. Curvature identities are of
fundamental importance for understanding the geometry of almost
Hermitian manifolds. The following curvature identities are used in
various studies e.g.(\cite{Graya,Graylv}):
\begin{enumerate}
\item $R(X,Y,Z,U)=R(X,Y,JZ,JU),$

\item  $R(X,Y,Z,U)=R(JX,JY,Z,U)+R(JX,Y,JZ,U)$

               $\quad\quad\quad\quad\quad\quad$$+R(JX,Y,Z,JU),$

\item $R(X,Y,Z,U)=R(JX,JY,JZ,JU)$.
\end{enumerate}
Let $AH_{i}$ denote the subclass of the class $AH$ of almost
Hermitian manifolds satisfying the curvature identity $(i).$,
$i=1,2,3.$ We know that
$$AH_{1}\subset AH_{2}\subset AH_{3}\subset AH,$$
from  \cite{Graya}. Some authors call $AH_{1}$-manifold as a
\emph{para-K\"{a}hlerian manifold} and call $AH_{3}$-manifold as an
$RK$-\emph{manifold} (\cite{Van}). An almost Hermitian manifold $M$
is called \emph{K\"{a}hlerian} if $\nabla_{X}J=0$ for all
$X\in\chi(M)$ and \emph{nearly K\"{a}hlerian (almost Tachibana)} if
$(\nabla_{X}J)X=0$ for all $X\in\chi(M).$ It is well-known that a
K\"{a}hlerian manifold is $AH_{1}$-manifold and  a nearly
K\"{a}hlerian manifold (non para-K\"{a}hlerian) manifold
is $AH_{2}$-manifold, see (\cite{Graylv,Van}).\\

A  two-dimensional linear subspace of a tangent space $T_{p}(M)$ is
called a \emph{plane section}. A plane section $\sigma$ is said to
be \emph{holomorphic} (resp.\emph{ anti-holomorphic} or
\emph{totally real}) if $J\sigma=\sigma$ (resp. $J\sigma\bot\sigma$)
(\cite{BChe}, \cite{Yama}). The sectional curvature $K$ of $M$ which
is determined by orthonormal vector fields $X$ and $Y$ is given by
$K(X,Y)=R(X,Y,X,Y).$ The  sectional curvature of $M$ restricted to a
holomorphic (resp. an anti-holomorphic) plane $\sigma$ is called
\emph{holomorphic} (resp. \emph{anti-holomorphic}) \emph{sectional
curvature}. If the holomorphic (resp. anti-holomorphic) sectional
curvature at each point $p\in M$ does not depend on $\sigma$, then
$M$ is said to be \emph{ pointwise constant holomorphic} (resp.
\emph{pointwise constant anti-holomorphic) sectional curvature}. A
connected Riemannian (resp. K\"{a}hlerian) manifold of (global)
constant sectional curvature (resp. of constant holomorphic
sectional curvature) is called a \emph{real space form} (resp. a
\emph{complex space form}) (\cite{BChe,Vanh,Kas,Yan}). The following
useful notion was defined by A. Gray in \cite{Gray}.
\begin{definition} Let $M$ be an almost Hermitian manifold. Then $M$ is said
to be of \emph{constant type} at  $p\in M$ provided that for all
$X\in T_{p}(M)$, we have $\lambda(X,Y)=\lambda(X,Z)$ whenever the
planes  $span\{X,Y\}$ and $span\{X,Z\}$ are  anti-holomorphic and
$g(Y,Y)=g(Z,Z)$, where the function $\lambda$ is defined by
$\lambda(X,Y)=R(X,Y,X,Y)-R(X,Y,JX,JY)$. If this holds for all $p\in
M$, then we say that  $M$ has \emph{(pointwise) constant type}.
Finally, if for $X,Y\in\chi(M)$ with $g(X,Y)=g(JX,Y)=0,$ the value
$\lambda(X,Y)$ is constant whenever $g(X,X)=g(Y,Y)=1$,
 then we say that $M$ has \emph{global constant type.}
\end{definition}
It follows that any $AH_{1}$-manifold has global vanishing constant type from Definition 2.1.\\

Let $M$ be a $2m$-dimensional K\"{a}hlerian manifold, for all $X,Y,Z,U\in T_{p}(M)$ and $p\in M,$
the Bochner curvature tensor $B$ \cite{Kassab} is defined by\\

$B(X,Y,Z,U)=R(X,Y,Z,U)-L(Y,Z)g(X,U)+L(Y,U)g(X,Z)$

$\quad\quad\quad\quad\quad\quad\quad$$-L(X,U)g(Y,Z)+L(X,Z)g(Y,U)-L(Y,JZ)g(X,JU)$

$\quad\quad\quad\quad\quad\quad\quad+L(Y,JU)g(X,JZ)-L(X,JU)g(Y,JZ)+L(X,JZ)g(Y,JU)$

$\quad\quad\quad\quad\quad\quad\quad+2L(Z,JU)g(X,JY)+2L(X,JY)g(Z,JU),$\\

where $L=\frac{\varrho}{2(m+2)} -\frac{\tau}{8(m+1)(m+2)}g,$
$\varrho$ is the \emph{Ricci tensor} and $\tau$ is the \emph{scalar
curvature} of $M.$ It is well-known that the Bochner curvature
tensor is a complex analogue of the Weyl conformal curvature tensor
\cite{Yan} on a Riemannian manifold.
\subsection{Submanifolds of a Riemannian manifold}
Let $N$ be a submanifold of a Riemannian manifold $M$ with a
Riemannian metric $g.$ Then Gauss and Weingarten formulas are
respectively given by $\nabla_{X}Y=\hat{\nabla}_{X}Y+h(X,Y)$  and
$\nabla_{X}\xi=-A_{\xi}X+\nabla_{X}^{\bot}\xi$ for all
$X,Y\in\chi(N)$ and $\xi\in\chi^{\bot}(N)$. Here $\nabla,
\hat{\nabla},$ and $\nabla^{\bot}$ are respectively the Riemannian,
induced Riemannian, and induced normal connection in $M, N,$ and the
normal bundle $\chi^{\bot}(N)$ of $N$, and $h$ is the second
fundamental form related to shape operator $A$ corresponding to the
normal vector field $\xi$ by $g(h(X,Y),\xi)=g(A_{\xi}X,Y).$ A
submanifold $N$ is said to be \emph{totally geodesic} if its second
fundamental form identically vanishes: $h=0,$ or equivalently
$A_{\xi}=0.$ We say that $N$ is \emph{totally umbilical} submanifold
in $M$ if for all $X,Y\in\chi(N),$ we have
\begin{equation}
\label{e1}
\begin{array}{c}
h(X,Y)=g(X,Y)\eta
\end{array},
\end{equation}
where $\eta\in\chi^{\bot}(N)$ is the mean curvature vector field of
$N$ in $M$. A vector field $\xi\in\chi^{\bot}(N)$ is said to be
\emph{parallel} if $ \nabla^{\bot}_{X}\xi=0$ for each $X\in\chi(N)$.
The Codazzi equation is given by
\begin{equation}
\label{e1}
\begin{array}{c}
(R(X,Y)Z)^{\bot}=(\nabla_{X}h)(Y,Z)-(\nabla_{Y}h)(X,Z)
\end{array},
\end{equation}
for all $X,Y,Z\in\chi(N),$ where $^{\bot}$ denotes the normal
component and the covariant derivative of $h$ denoted by
$\nabla_{X}h$ is defined by
\begin{equation}
\label{e1}
\begin{array}{c}
(\nabla_{X}h)(Y,Z)=\nabla^{\bot}_{X}(h(Y,Z))-h(\hat{\nabla}_{X}Y,Z)-h(Y,\hat{\nabla}_{X}Z)
\end{array},
\end{equation}
for all $X,Y,Z\in\chi(N)$ (\cite{BChe,Gol,Ha,Yama}).
\subsection{Anti-invariant submanifolds of an almost Hermitian manifold}

Let $M$ be a $2m$-dimensional almost Hermitian manifold endowed with
an almost complex structure $J$ and a Hermitian metric $g.$ An
$n$-dimensional Riemannian manifold $N$ isometrically immersed in
$M$ is called an \emph{anti-invariant submanifold} of $M$ (or
\emph{totally real submanifold} of $M$) if $JT_{p}(N)\subset
T_{p}(N)^{\bot}$ for each point $p$ of  $N$. Then we have $m\geq n$
(\cite{Yan}).
\section{The axiom of anti-invariant 2-spheres}
S. Yamaguchi and M. Kon \cite{Yama} introduced the axiom of
anti-invariant 2-spheres as: An almost Hermitian manifold $M$ is
said to satisfy the axiom of anti-invariant 2-spheres, if for each
point $p\in M$ and each anti-holomorphic $2$-plane $\sigma$ of the
tangent space $T_{p}(M)$, there exists a $2$-dimensional totally
umbilical anti-invariant submanifold $N$ such that $p\in N$ and
$T_{p}(N)=\sigma.$ They proved in Theorem 1(\cite{Yama}) that a
K\"{a}hlerian manifold with the axiom of anti-invariant 2-spheres is
a complex space form. Here, we give a generalization of Theorem
1(\cite{Yama}). From now on, we shall assume that all manifolds are
connected throughout this study.\\

We shall need the following Lemma for the proof of the main Theorem
3.1 below.
\begin{lemma}(\cite{Ho}) Let $N$ be an anti-invariant submanifold of a nearly
K\"{a}hlerian manifold $M$. Then
\begin{equation}
\label{e1}
\begin{array}{c}
A_{JZ}X=A_{JX}Z
\end{array}
\end{equation}
holds for any two vectors $X$ and $Z$ tangent to $N$, where $A$ is
the shape operator of $N$.
\end{lemma}
Let $N$ be given as in Lemma 3.1, then
\begin{equation}
\label{e1}
\begin{array}{c}
g(A_{JZ}X,Y)=g(A_{JX}Z,Y)=g(A_{JY}X,Z)
\end{array},
\end{equation}
for any vectors $X,Y$ and $Z$ tangent to $N$, from (3.1). This
equation is equivalent to
\begin{equation}
\label{e1}
\begin{array}{c}
g(h(X,Y),JZ)=g(h(Z,Y),JX)=g(h(X,Z),JY)
\end{array},
\end{equation}
where $h$ is the second fundamental form of $N$.
\begin{theorem}
Let $M$ be a nearly K\"{a}hlerian manifold of dimension $2m\geq6.$
If $M$ satisfies the axiom of anti-invariant 2-spheres, then $M$ is
a space of constant holomorphic sectional curvature.
\end{theorem}
\begin{proof}
Let $p$ be any point of $M$. Let $X$ and $Y$  be any orthonormal
vectors of $T_{p}(M)$ spanning an anti-holomorphic 2-plane $\sigma$.
By the axiom of anti-invariant 2-spheres, there exists a
two-dimensional  totally umbilical anti-invariant submanifold $N$
such that $p\in N$ and $T_{p}(N)=\sigma.$ Then,  we have
\begin{equation}
\label{e1}
\begin{array}{c}
(R(X,Y)Y)^{\bot}=\nabla^{\bot}_{X}\eta
\end{array},
\end{equation}
with the help of (2.1) and (2.3) from (2.2), where $\eta$ is the
mean curvature vector field of $N$ in $M$. We find
\begin{equation}
\label{e1}
\begin{array}{c}
R(X,Y,Y,JX)=g(\nabla^{\bot}_{X}\eta,JX)
\end{array},
\end{equation}
from (3.4), since $JX$ is normal to $N.$ If we put $Y=Z$ into (3.3),
we obtain
\begin{equation}
\label{e1}
\begin{array}{c}
g(\eta,JX)=0
\end{array},
\end{equation}
 by using (2.1). Similarly, we also have $g(\eta,JY)=g(\eta,JZ)=0$, for any vectors $X,Y,$ and $Z$ tangent to $N.$
If we differentiate the equation (3.6) with respect to $X$, then we
have
\begin{equation}
\label{e1}
\begin{array}{c}
0=X[g(\eta,JX)]=g(\nabla^{\bot}_{X}\eta,JX)+g(\eta,\nabla^{\bot}_{X}JX)
\end{array}.
\end{equation}
Upon straightforward calculation, we see that
$g(\eta,\nabla^{\bot}_{X}JX)=0.$ Thus, we obtain
\begin{equation}
\label{e1}
\begin{array}{c}
g(\nabla^{\bot}_{X}\eta,JX)=0
\end{array},
\end{equation}
from (3.7). We get
\begin{equation}
\label{e1}
\begin{array}{c}
R(X,Y,Y,JX)=0
\end{array}
\end{equation}
by combining (3.8) with (3.5). For all orthonormal vectors $X,Y\in
T_{p}(M)$ with $g(X,JY)=0,$  we derive
\begin{equation}
\label{e1}
\begin{array}{c}
R(JY,Y,Y,JX)=0,
\end{array}
\end{equation}
by replacing $X$ by $\frac{1}{\sqrt{2}}(X+JY)$ in (3.9). In this
case, it follows that $M$ has pointwise constant holomorphic
sectional curvature, using (3.10) and Lemma 1(\cite{Kas}). We
conclude that $M$ is a space of constant holomorphic sectional
curvature by Theorem 6(\cite{He}).
\end{proof}
\section{The axiom of co-holomorphic $(2n+1)$-spheres}

Let $M$ be a $2m$-dimensional almost Hermitian manifold. L. Vanhecke
\cite{Vanh} defined a co-holomorphic $(2n+1)$-plane as a
$(2n+1)$-plane containing a holomorphic $2n$-plane for the manifold
$M$. It is not difficult to see that a co-holomorphic $(2n+1)$-plane
contains an anti-holomorphic $(n+1)$-plane and that $1\leq n\leq
m-1$. He also gave the axiom of co-holomorphic $(2n+1)$-spheres as:
A $2m$-dimensional almost Hermitian manifold $M$ which is said to
satisfy the axiom of co-holomorphic $(2n+1)$-spheres, if for each
point $p\in M$ and each co-holomorphic $(2n+1)$-plane $\sigma$ of
the tangent space $T_{p}(M)$, there exists a $(2n+1)$-dimensional
totally umbilical submanifold $N$ such that $p\in N$ and
$T_{p}(N)=\sigma.$
He studied this axiom for $AH_{3}$-manifolds and obtained several results.\\

Now, we study this axiom for larger classes of almost Hermitian
manifolds.
\begin{lemma}
Let $M$ be an almost Hermitian manifold of dimension $2m\geq4.$ If
$M$ satisfies the axiom of co-holomorphic $(2n+1)$-spheres, then we
have
\begin{equation}
\label{e1}
\begin{array}{c}
\lambda(X,Y)=K(X,Y)
\end{array},
\end{equation}
for all orthonormal vectors $X,Y\in T_{p}(M)$ with $g(X,JY)=0,$
where $\lambda(X,Y)=R(X,Y,X,Y)-R(X,Y,JX,JY)$ and $K$ denotes
anti-holomorphic sectional curvature.
\end{lemma}
\begin{proof}
Let $p$ be an arbitrary point of $M$. Let $X$ and $Y$  be any
orthonormal vectors in $T_{p}(M)$ with $g(X,JY)=0,$ that is, they
span an anti-holomorphic plane. Consider the co-holomorphic
$(2n+1)$-plane $\sigma$ containing $X,JX,$ and $Y$ such that $JY$ is
normal to $\sigma$. By the axiom of co-holomorphic $(2n+1)$-spheres,
there exists a $(2n+1)$-dimensional totally umbilical submanifold
$N$ such that $p\in N$ and $T_{p}(N)=\sigma.$ Then, we have
\begin{equation}
\label{e1}
\begin{array}{c}
(R(X,Y)JX)^{\bot}=0
\end{array}.
\end{equation}
 with the help of
(2.1) and (2.3), from (2.2). We get
\begin{equation}
\label{e1}
\begin{array}{c}
R(X,Y,JX,JY)=0
\end{array},
\end{equation}
from (4.2), since $JY$ is normal to $N$. Thus, our assertion follows
from Definition 2.1 and (4.3).
\end{proof}
Now, we are ready to prove our second main result.
\begin{theorem}
Let $M$ be an  almost Hermitian manifold of dimension $2m\geq6.$ If
$M$ satisfies the axiom of co-holomorphic $(2n+1)$-spheres for some
$n$, then $M$ has pointwise constant type $\alpha$ if and only if
$M$ has pointwise constant anti-holomorphic sectional curvature
$\alpha$.
\end{theorem}
\begin{proof}
Let $M$ be an  almost Hermitian manifold of dimension $2m\geq6$
satisfying the axiom of co-holomorphic $(2n+1)$-spheres for some
$n$. If $M$ has pointwise constant type, that is, $M$ has constant
type at $p$, for all $p\in M$. Then, for all $X,Y,Z\in T_{p}(M)$
whenever the planes $span\{X,Y\}$ and $span\{X,Z\}$ are
anti-holomorphic and $g(Y,Y)=g(Z,Z),$ we have
\begin{equation}
\label{e1}
\begin{array}{c}
\lambda(X,Y)=\lambda(X,Z)
\end{array}.
\end{equation}
Here, we can assume that $g(Y,Y)=g(Z,Z)=1.$ Thus, for all
orthonormal vectors $X,Y,Z\in T_{p}(M)$ with $g(X,JY)=g(X,JZ)=0$, we
get
\begin{equation}
\label{e1}
\begin{array}{c}
K(X,Y)=K(X,Z)
\end{array},
\end{equation}
from Lemma 4.1.\\
On the other hand, we can choose a unit vector $U$ in
$(span\{X,JX\})^{\bot}\cap(span\{Z,JZ\})^{\bot},$  since the
dimension of $M$ is greater than 6. Then, we have
\begin{equation}
\label{e1}
\begin{array}{c}
K(X,U)=K(X,Z)
\end{array},
\end{equation}
from (4.5). This implies that the sectional curvature is same for
all anti-holomorphic sections which contain any given vector $X.$
Hence, we write
\begin{equation}
\label{e1}
\begin{array}{c}
K(X,Y)=K(Y,Z)=K(Z,U)
\end{array}.
\end{equation}
Therefore, we find
\begin{equation}
\label{e1}
\begin{array}{c}
K(X,Y)=K(Z,U)
\end{array},
\end{equation}
for all $X,Y,Z,U\in T_{p}(M)$ whenever the planes $span\{X,Y\}$ and
$span\{Z,U\}$ are anti-holomorphic. It follows that the sectional
curvature is same for all anti-holomorphic sections at $p\in M.$
Namely, $M$ has pointwise constant
anti-holomorphic sectional curvature.\\
Conversely, let $M$ be of pointwise constant anti-holomorphic
sectional curvature and let $p$ be any point of $M$. Then for all
orthonormal vectors $X,Y,Z\in T_{p}(M)$ with $g(X,JY)=g(X,JZ)=0,$
($span\{X,Y\}$ and $span\{X,Z\}$ are anti-holomorphic planes and
$g(X,X)=g(Y,Y)=g(Z,Z)=1$), we have
\begin{equation}
\label{e1}
\begin{array}{c}
K(X,Y)=K(X,Z)
\end{array}.
\end{equation}
By Lemma 4.1, we get
\begin{equation}
\label{e1}
\begin{array}{c}
\lambda(X,Y)=\lambda(X,Z)
\end{array},
\end{equation}
for all orthonormal vectors $X,Y,Z\in T_{p}(M)$  whenever the planes
$span\{X,Y\}$ and $span\{X,Z\}$ are anti-holomorphic. It is not
difficult to see that (4.10) also holds in the case
$g(Y,Y)=g(Z,Z)\neq1.$ It follows that $M$ has constant type at $p.$
Additionally, if the constant value of $\lambda(X,Y)$ equals
$\alpha$, then the pointwise constant anti-holomorphic sectional
curvature $K$ must be $\alpha$, because of Lemma 4.1.
\end{proof}
We remark that above technical method was used also in Theorem 3.4(\cite{Tas}).\\

Next, we give some applications of Theorem 4.1.
\begin{theorem}
Let $M$ be a $2m$-dimensional almost Hermitian manifold with
pointwise constant type
$\alpha$ and $m\geq3.$ If $M$ satisfies the axiom of co-holomorphic $(2n+1)$-spheres for some $n$, then\\
\textbf{i)} $M$ is a space of constant curvature $\alpha$ and $M$ has global constant type $\alpha$,\\
\textbf{ii)} $M$ is an $AH_{2}$-manifold.
\end{theorem}
\begin{proof} Let $p$ be any point of $M$. Let $X$ and $Y$  be any orthonormal
vectors in $T_{p}(M)$ with $g(X,JY)=0.$  Consider the co-holomorphic
$(2n+1)$-plane $\sigma$ containing $X,JX,$ and $JY$ such that $Y$ is
normal to $\sigma$. By the axiom of co-holomorphic $(2n+1)$-spheres,
there exists a $(2n+1)$-dimensional totally umbilical submanifold
$N$ such that $p\in N$ and $T_{p}(N)=\sigma.$ Then, we have
\begin{equation}
\label{e1}
\begin{array}{c}
(R(X,JX)JX)^{\bot}=0
\end{array}.
\end{equation}
with the help of (2.1) and (2.3), from (2.2). Hence, we get
\begin{equation}
\label{e1}
\begin{array}{c}
R(X,JX,JX,Y)=0
\end{array},
\end{equation}
for all orthonormal vectors $X,Y\in T_{p}(M)$ with $g(X,JY)=0,$
since $Y$ is normal to $N$. It follows that $M$ is an
$AH_{3}$-manifold with pointwise constant holomorphic sectional
curvature using (4.12) and Lemma 1(\cite{Kas}) together with Lemma
3(\cite{Kas}). On the other hand, we have that $M$ has constant
anti-holomorphic sectional curvature $\alpha$ at $p,$ from  Theorem
4.1. In this case, we obtain that the constant holomorphic sectional
curvature $H$ of $M$ is $\alpha$ at $p$ from Theorem 5(\cite{Van}).
By using Theorem 4(\cite{Van})in Theorem 2(\cite{Van}), we obtain
\begin{equation}
\label{e1}
\begin{array}{c}
K(X,Y)=\alpha
\end{array},
\end{equation}
for all orthonormal vectors $X,Y\in T_{p}(M),$ where
$K(X,Y)=R(X,Y,X,Y)$ is sectional curvature. It is not difficult to
see that (4.13) is also true for all $X,Y\in T_{p}(M).$ By the
well-known Schur's theorem(\cite{Yan})
it follows that $M$ is a space of constant curvature $\alpha$ and $M$ has global constant type.\\
Now, we prove the part \textbf{ii)}. From the part \textbf{i)},
automatically, both holomorphic and anti-holomorphic sectional
curvature equal to $\alpha$. In which case, it follows from  Theorem
3(\cite{Gan}) that $M$ is an $AH_{2}$-manifold.
\end{proof}
\begin{remark}
Theorem 4.2 without part \textbf{ii)} was also obtained by O.T.
Kassabov in (\cite{Kass}) with different approach. It is a
generalization of Theorem 1(\cite{Vanh}) concerning
$AH_{3}$-manifolds.
\end{remark}
By Theorem 4.1, Theorem 4.2 and Theorem 5(\cite{Van}), we have the
following result which is a generalization of Corollary
1(\cite{Vanh}) concerning  $AH_{3}$-manifolds.
\begin{corollary} Let $M$ be a $2m$-dimensional almost Hermitian manifold with
vanishing constant type and $m\geq3.$ If $M$ satisfies the axiom of
co-holomorphic $(2n+1)$-spheres for some $n$, then $M$ is a flat
$AH_{2}$-manifold.
\end{corollary}
We end this paper by giving a result related to the Bochner
curvature tensor of a K\"{a}hlerian manifold satisfying the axiom of
co-holomorphic $(2n+1)$-spheres.
\begin{theorem}
Let $M$ be a K\"{a}hlerian manifold of dimension $2m\geq6.$ If $M$
satisfies the axiom of co-holomorphic $(2n+1)$-spheres for some $n$,
then $M$ has a vanishing Bochner curvature tensor.
\end{theorem}
\begin{proof} Let $p$ be any point of $M$. Let $X,Y,$ and $Z$  be any unit vectors
of $T_{p}(M)$, which span an anti-holomorphic 3-plane, that is,
$g(X,Y)=g(X,Z)=g(Y,Z)=0,$ and $g(X,JY)=g(X,JZ)=g(Y,JZ)=0.$ Consider
the co-holomorphic $(2n+1)$-plane $\sigma$ containing $X,JX,$ and
$Y$ such that $Z$ is normal to $\sigma$. By the axiom of
co-holomorphic $(2n+1)$-spheres,  there exists a
$(2n+1)$-dimensional totally umbilical submanifold $N$ such that
$p\in N$ and $T_{p}(N)=\sigma.$ Then, we have
\begin{equation}
\label{e1}
\begin{array}{c}
(R(X,JX)Y)^{\bot}=0
\end{array},
\end{equation}
and
\begin{equation}
\label{e1}
\begin{array}{c}
(R(X,Y)JX)^{\bot}=0
\end{array}.
\end{equation}
 with the help of (2.1) and
(2.3) from (2.2). For all unit vectors $X,Y,Z\in T_{p}(M)$, which
span an anti-holomorphic 3-plane, we respectively get
\begin{equation}
\label{e1}
\begin{array}{c}
R(X,JX,Y,Z)=0
\end{array}
\end{equation}
and
\begin{equation}
\label{e1}
\begin{array}{c}
R(X,Y,JX,Z)=0
\end{array},
\end{equation}
from (4.14) and (4.15), since $Z$ is normal to $N$. Thus, our
assertion follows from (4.16), (4.17) and Lemma(\cite{Kassa}).
\end{proof}

\end{document}